\documentclass{amsart}

\usepackage{amsmath, latexsym,  amssymb, amsthm, amscd, comment, enumerate}
\usepackage[all]{xy}
\usepackage[british]{babel}
\usepackage{url}

\usepackage{enumitem}

\DeclareMathOperator{\codim}{codim}

\newcommand{\sotto}{B}

\newcommand{\cod}{\mathrm{cod} \,}

\newcommand{\QQ}{\mathbb{Q}}
\newcommand{\RR}{\mathbb{R}}

\newcommand{\PP}{\mathbb{P}}
\newcommand{\CC}{\mathbb{C}}

\renewcommand{\epsilon}{\varepsilon}

\newtheorem{thm}{Theorem}[section]

\newtheorem{con}[thm]{Conjecture}
\newtheorem{propo}[thm]{Proposition}
\newtheorem{lem}[thm]{Lemma}

\newtheorem{D}[thm]{Definition}

\title[The elliptic torsion anomalous conjecture in codimension~$2$]{The elliptic torsion anomalous conjecture\\ in codimension~$2$}

\author{Patrik Hubschmid}

\author{Evelina Viada}

\thanks{All authors are supported by the DFG}



\begin{document}

\begin{abstract}
The torsion anomalous conjecture states that for any  variety $V$ in an abelian variety there are only finitely many maximal  $V$-torsion anomalous varieties. We complement previous results and prove this conjecture for $V$ of codimension $2$ in a product $E^N$  of any elliptic curve $E$. This was known only when $E$ has CM. We also give an effective upper bound for the normalized height of these maximal $V$-torsion anomalous varieties.
\end{abstract}

\maketitle

\section{Introduction}
In this article, by a \emph{variety}, we always mean an \emph{irreducible} algebraic variety defined over the field $\overline{\QQ}$ of algebraic numbers and, by a \emph{point}, we mean a $\overline{\QQ}$-valued point. In particular, we consider subvarieties $V$ of an abelian variety $A$, both defined over $\overline{\QQ}$.
 
We say that $V \subset A$ is a \emph{translate}, respectively a \emph{torsion variety}, if it is a translate of a proper abelian subvariety of $A$ by a point, respectively by a torsion point.

A subvariety $V \subset A$  is \emph{transverse}, respectively \emph{weak-transverse}, if it is not contained in any translate, respectively in any torsion variety.

Of course a torsion variety is in particular a translate, and a transverse variety is weak-transverse. In addition, transverse implies non-translate, and weak-transverse implies non-torsion.

\medskip

In a well known paper,  
Bombieri, Masser and Zannier \cite{BMZ1} introduced  the notions of anomalous and torsion anomalous varieties and stated a general conjecture which implies several classical conjectures such as the Manin-Mumford and the Mordell-Lang Conjectures.  For simplicity, we use here the definition given in  \cite{CVV_rendiconti} by Checcoli,  Veneziano and the second author; unlike in  \cite{BMZ1}, for us points  can be torsion anomalous, but not anomalous. 

We have the following definitions.

\begin{D}
Let $V$ be a subvariety of an abelian variety. A subvariety $Y$ of $V$ is a \emph{$V$-torsion anomalous} variety if 
\begin{itemize}
\item[-] $Y$ is an irreducible component of $V\cap B$ with $B$ a torsion variety;
\item[-]  the dimension of $Y$ is larger than expected, i.e.,  $$\codim Y < \codim V +  \codim  B,$$ where $\codim$ indicates the codimension.
\end{itemize}

The torsion variety $B$ is \emph{minimal} for $Y$ if it satisfies the above conditions and has minimal dimension.
The \emph{relative codimension} of $Y$ is the codimension of $Y$ in its minimal $B$.

We say that $Y$ is a \emph{maximal $V$-torsion anomalous} variety if it is $V$-torsion anomalous and it is not contained in any $V$-torsion anomalous variety of strictly larger dimension.

\end{D}

In the Torsion Openness Conjecture \cite{BMZ1}, Bombieri, Masser and Zannier conjecture that the complement of the set of the torsion anomalous varieties of positive dimension is open. In addition, in the Torsion Finiteness Conjecture, they claim that there are only finitely many maximal torsion anomalous points. In   \cite{CVV_rendiconti} these conjectures are  reformulated in a slightly more general statement.
 \begin{con}[Torsion Anomalous Conjecture]
\label{TORAN}
Let $V$ be a variety embedded in an abelian variety. Then there are only finitely many maximal $V$-torsion anomalous varieties.
\end{con}


It is well known that the Torsion Anomalous Conjecture implies several  classical theorems, such as   the Manin-Mumford Conjecture, proven by Raynaud~\cite{raynaud}. However, we are not giving any new proof of this important result, on the contrary we use it in our proof.
\begin{thm}[Manin-Mumford Conjecture] \label{con:MM}
Let $V$ be a proper subvariety of an abelian variety. Then $V$ contains only finitely many torsion varieties which are maximal in $V$. The sum of their degree is effectively bounded in terms of $\deg V$.
\end{thm}
For more details on the Torsion Anomalous Conjecture  and its implication one can look at the survey article \cite{VPisa}.

Clearly if $V$ is not weak-transverse then it is itself  $V$-torsion anomalous, and the conjecture holds.
Also, for a hypersurface, Conjecture~\ref{TORAN} is clearly true. Indeed, the intersection of a torsion variety $\sotto$ with a hypersurface is either  $\sotto$ itself or it has the right dimension $\dim \sotto-1$. So the only maximal $V$-torsion anomalous varieties are maximal torsion varieties contained in $V$; thus finitely many by the Manin-Mumford Conjecture.

The Torsion Anomalous Conjecture is known for curves  in a product of elliptic curves (Viada  \cite{ioant}), in abelian varieties with CM (R\'emond \cite{remond09}), in abelian varieties with a positive density of ordinary primes (Galateau \cite{galateau} and Viada \cite{IJNT}) and in abelian varieties (Habegger and Pila \cite{PP}); in a torus  (Maurin  \cite{Maurin}).

Among other results, Bombieri, Masser and Zannier in \cite[Theorem 1.7]{BMZ1} prove Conjecture~\ref{TORAN} for a  variety $V$ of codimension $2$ in $\mathbb{G}^n_m$, where in the proof they use the Ax theorem.

In \cite{CVV_rendiconti}  Checcoli,  Veneziano and  the second author prove Conjecture~\ref{TORAN} for a variety $V$ of codimension $2$ in a product $E^N$ where $E$ is an elliptic curve with CM, avoiding the Ax theorem. Instead, this result relies on a Lehmer type bound which is not available for non CM elliptic curves.
 
In this article, we extend this result by giving a proof of Conjecture~\ref{TORAN} for a subvariety $V$ of codimension $2$ in a power $E^N$ of any elliptic curve $E$ (CM or non-CM). Our main result is the following theorem.

\begin{thm}
\label{th:main}
For a subvariety $V$ of a power $E^N$ of an elliptic curve $E$ with $\codim V = 2$,  there are only finitely many maximal $V$-torsion anomalous varieties $Y$. Furthermore, the normalized height $h(Y)$  is effectively bounded in terms of $E$, $N$ and $V$.
\end{thm}
For reducible varieties it is sufficient to apply the theorem to each irreducible component.

Our result is a complement to the analogous result in~\cite{CVV_rendiconti} where Checcoli, Veneziano and Viada prove Conjecture~\ref{TORAN} with the further assumption that  $E$ has CM. In particular, we already know by \cite[Theorem 5.1]{CVV_rendiconti} that there are only finitely many maximal $V$-torsion anomalous varieties which are not  translates. 
To handle the case of translates, for the non CM case, we cannot use a Lehmer type bound like in~\cite{CVV_rendiconti}. Instead, for points we use the approximation process used by Checcoli, Veneziano and Viada in~\cite[\S 3]{CVV_trans} where one constructs a translate of a given dimension with controlled degree and height through a given point in a torsion variety of dimension one. To bound the normalized height for translates of positive dimension we use a more complicated diophantine approximation generalizing the just mentioned process. Finally, to prove the finiteness of such anomalous varieties we use an induction argument and we also use a result of  Viada \cite{Evelina_IMRN} on the non density of certain  points of bounded height.

\section{Preliminaries}
As already mentioned in the introduction, all varieties in this article are defined over the algebraic numbers $\overline{\QQ}$ and we identify $V$ with $V(\overline{\QQ})$.

We denote by $E$ an elliptic curve over $\overline{\QQ}$ together with a fixed Weierstrass equation
\[ E: y^2 = x^3 + Ax + B \]
with $A, B $ algebraic integers. We consider $E$ as a subvariety of $\PP^2$ defined by this equation.
Let $N$ be a positive integer. The $N$-th power $E^N \hookrightarrow (\PP^2)^N$ of the natural inclusion $E \hookrightarrow \PP^2$ composed with the Segre embedding $(\PP^2)^N \hookrightarrow \PP^{3^N-1}$ gives a closed immersion
\[ E^N \hookrightarrow \PP^m \]
for
\[ m :=  3^N-1. \]
We consider a subvariety $X \subset E^N$ as a subvariety of $\PP^m$ via the above embedding and compute its degree $\deg X$ as a subvariety of $\PP^m$.

\subsection{Abelian subvarieties and Lie subalgebras} \label{subsec:absubvarieties}
We recall the analytic uniformization of the elliptic curve $E$.  There is a unique lattice $\Lambda \subset \CC$ such that the map $\CC \rightarrow \PP^2(\CC)$ given by
\[  z \longmapsto \left\{ \begin{array}{lll}
[\wp_\Lambda(z) : \wp'_\Lambda (z) : 1 ] & , & z \not\in \Lambda \\{}
[0:1:0] &, & z \in \Lambda
\end{array}    \right. \]
induces an isomorphism $\CC / \Lambda \stackrel{\sim}{\rightarrow} E(\CC) \subset \PP^2(\CC)$ of complex Lie groups (see~\cite[Theorem VI.5.1]{Sil}). Here $\wp_\Lambda$ denotes the Weierstrass $\wp$-function associated to the lattice $\Lambda$. The $N$-th power of this isomorphism gives the analytic uniformization
\[ \phi: \CC^N / \Lambda^N \stackrel{\sim}{\rightarrow} E^N(\CC) \]
of $E^N$. Note that $\CC^N$ can be identified with the Lie algebra of $E(\CC)$ and the composition of the canonical projection $\pi: \CC^N \rightarrow \CC^N / \Lambda^N$ with $\phi$ can be identified with the exponential map of the Lie group $E^N(\CC)$. By~\cite[8.9.8]{BiGi}, the set of abelian subvarieties of $E^N$ is in natural bijection with the set of complex subvectorspaces $W \subset \CC^N$ for which $W \cap \Lambda^N$ is a lattice of full rank in $W$. Such a $W$ corresponds to the abelian subvariety $B$ of $E^N$ with $B(\CC) = (\phi \circ \pi)(W)$ and we identify $W$ with the Lie algebra of $B$.

We define the \emph{orthogonal complement} $B^\perp$ of an abelian subvariety $B \subset E^N$ with Lie algebra $W \subset \CC^N$ as the abelian subvariety with Lie algebra $W^\perp$ where $W^\perp$ denotes the orthogonal complement of $W$ with respect to the canonical Hermitian structure of $\CC^N$. Note that it can easily be shown that $W^\perp \cap \Lambda^N$ is a lattice of full rank in $W^\perp$ using a similar argumentation as in~\cite[8.9.8]{BiGi}.

We remark that by Masser and W\"ustholz~\cite[Lemma 1.4]{Masserwustholz}, we have $E^N = B + B^\perp$ and that the intersection of an abelian subvariety $B$ and its orthogonal complement $B^\perp$ is finite, indeed we have
\[ |B \cap B^\perp| \ll (\deg B)^2. \]

Moreover, for any two abelian subvarieties $B_1$ and $B_2$ of $E^N$,  the short exact sequence 
$$ 0 \to B_1\cap B_2 \to B_1\times B_2 \to B_1+B_2 \to 0$$ implies 
\[ \frac{\deg (B_1 + B_2)}{(\dim (B_1 + B_2))!} \leq \frac{\deg B_1}{(\dim B_1)!} \cdot \frac{\deg B_2}{(\dim B_2)!}, \]
see \cite[Lemma 1.2]{Masserwustholz}. Hence we have
\begin{equation} \label{ineq:sum} \deg(B_1 + B_2) \leq c_6 \deg B_1 \deg B_2 \end{equation}
with $c_6 := 2^N$.

 Finally, we remark that, for an abelian subvariety $B$ of $E^N$, there is a surjective morphism $\phi_B: E^N \rightarrow E^r$ such that $\ker \phi_B = B + \tau$ for $\tau$ a set of torsion points of absolutely bounded cardinality (by~\cite[Lemma 1.3]{Masserwustholz}). From the above description and Minkowski Theorem,  the morphism $\phi_B$ can be described by a matrix in $\mathrm{Mat}_{r \times N} (\mathrm{End}(E))$ with rows $u_1,\ldots,u_r \in (\mathrm{End}(E))^N$ such that
\[ \prod_{i=1}^r \|u_i\|_2^2 \ll \deg B \ll \prod_{i=1}^r \|u_i\|_2^2 \]
where the constants are effective and only depend on $N$ and $r$. 
Since each ball of given radius in $(\mathrm{End}(E))^N$ only contains finitely many elements, we conclude that there are only finitely many abelian subvarieties of $E^N$ of bounded degree.

\subsection{The normalized height of a subvariety}

We  consider the \emph{normalized height} $h(X)$ of a subvariety $X$  of $\PP^m$ as defined by Philippon in~\cite{patrice}. This is defined in terms of an appropriate height of the Chow form of $X$, which is a multihomogeneous polynomial with coefficients in $\overline{\QQ}$ defined uniquely up to a scalar multiple. 

\subsection{The arithmetic B\'ezout theorem}
One of the central theorems of arithmetic intersection is  the arithmetic B\'ezout theorem~\cite[Theorem 5.5.1 (iii)]{BGS}. This theorem plays a central role in our proof.

\begin{thm}
Let $X$ and $Y$ be subvarieties of $E^N$. If $Z_1, \ldots, Z_g$ are the irreducible components of $X \cap Y$, then we have
\[ \sum_{i=1}^g h(Z_i) \leq \deg(X) h(Y) + \deg(Y) h(X) + \frac{3^N \log 2}{2} \deg (X) \deg (Y). \]
\end{thm}

\subsection{Height functions for  points}
For points $P$ in $E^N$, we  use several different height functions. On the one hand, we use the \emph{absolute logarithmic Weil height} $h(P)$ of $P$ as a point in $\PP^m(\overline{\QQ})$. If the homogeneous coordinates $(P_0 : \cdots : P_m)$ of $P$ lie in a number field~$K$, then 
\[ h(P) = \sum_{v \in M_K} \frac{[K_v : \QQ_v]}{[K : \QQ]} \log \max_i \{|P_i|_v \} \]
where $M_K$ is the set of places of $K$. 

Note that  the absolute logarithmic Weil height $h(P)$ of a point $P$ differs from the normalized height $h(\{P\})$ of $\{P\}$ as a subvariety of $\PP^m$. To avoid confusion, we  write $h_2(P) := h(\{P\})$ for the normalized height of a point and use $h(P)$  for the absolute logarithmic Weil height. By~\cite[(3.1.6)]{BGS} we have
\[ h_2(P) = \sum_{v\ \mathrm{finite}} \frac{[K_v : \QQ_v]}{[K : \QQ]} \log \max_i \{|P_i|_v\} 
+ \sum_{v\ \mathrm{infinite}} \frac{[K_v : \QQ_v]}{[K : \QQ]} \log \left( \sum_i |P_i|_v^2 \right)^{1/2}
\]
and, by the norm estimates $\| x \|_\infty \leq \|x \|_2 \leq \sqrt{m+1} \|x \|_\infty$ on $\RR^{m+1}$, we have the inequalities
\[ h(p) \leq h_2 (p) \leq h(p) + \frac12 \log(m + 1). \]
On the other hand, we also consider the canonical \emph{N\'eron-Tate height} $\hat{h}$  induced by our fixed embedding  of $E^N$ in $\mathbb{P}^m$. 
It is non-negative and for $P = (P_1,\ldots,P_N) \in E^N$ we have $\hat{h}(P) = \sum_{i=1}^N \hat{h}(P_i)$ where $\hat{h}$ denotes the canonical N\'eron-Tate height on $E$ in $\mathbb{P}^2$.

The N\'eron-Tate height $\hat{h}(P)$ of a point $P$ of $E^N$ is related to the normalized height $h_2(P)$ via the inequalities
\begin{eqnarray}
h_2(P) & \leq & 3 \hat{h}(P) + c_1, \label{ineq:normalized_Neron1}  \\ 
\hat{h}(P) & \leq & \frac{h_2(P)}{3} + c_2 \label{ineq:normalized_Neron2} 
\end{eqnarray}
with some explicit positive constants $c_1 = c_1(E, N)$ and $c_2 = c_2(E, N)$ only depending on $E$ and $N$ (see~\cite[(8) and (10)]{CVV_trans} for details). 

\subsection{The essential minimum}
For a subvariety $X \subset E^N$, we consider the \emph{normalized essential minimum} $\mu(X)$ of $X$ defined by
\[ \mu(X) := \mathrm{inf} \{ \theta \in \RR_{\geq 0}\, : \, \{ P \in X \, : \, h_2(P) \leq \theta \}\ \text{is Zariski dense in}\ X \}. \]
By the theorem of successive minima~\cite[p. 161]{ZhangEquidistribution} used in Zhang's proof of the Bogomolov conjecture we have the inequalities
\begin{equation}
\mu(X) \leq \frac{h(X)}{\deg(X)} \leq (1 + \dim X) \mu(X), \label{ineq:zhang}
\end{equation} 
which we refer as \emph{Zhang's inequality}.

We will also use the \emph{N\'eron-Tate essential minimum} $\hat{\mu}(X)$ of $X$ defined by
\[ \hat{\mu}(X) := \mathrm{inf} \{ \theta \in \RR_{\geq 0}\, : \, \{ P \in X \, : \, \hat{h}(P) \leq \theta \}\ \text{is Zariski dense in}\ X \}. \]
We note that we have $\hat{\mu}(X) = 0$ for all torsion varieties $X$ in $E^N$ because each torsion variety contains a dense set of torsion points. By the Bogomolov conjecture, we furthermore know that $\hat{\mu}(X) > 0$ for all non-torsion subvarieties of $E^N$. By~\eqref{ineq:normalized_Neron1} and \eqref{ineq:normalized_Neron2} the two definitions of the essential minimum of a subvariety of $E^N$ are related via the inequalities
\begin{eqnarray}
\mu(X) & \leq & 3 \hat{\mu}(X) + c_1, \label{ineq:mu1} \\ 
\hat{\mu}(X) & \leq & \frac{\mu(X)}{3} + c_2. \label{ineq:mu2}
\end{eqnarray}

\subsection{The essential minimum of a translate}
We use a lemma due to Philippon~\cite{Philippon_preprint} to calculate the N\'eron-Tate essential minimum of a translate $Y = B + p \subset E^N$ of an abelian subvariety $B \subset E^N$ (see also~\cite[Lemma 7.2]{CVV_rendiconti} for details). 

\begin{lem} \label{lem:Philippon}
Let $B + p$ be a translate in $E^N$. Then $B+p=B + p^\perp$ with $p^\perp$ lying in the orthogonal complement of $B$ and
\[ \hat{\mu}(B + p^\perp) = \hat{h}(p^\perp).   \]
\end{lem}

\subsection{The auxiliary translate}
As mentioned in the introduction, part of the method used here relies on an approximation process for a point in a torsion variety. We need a translate  of given dimension and controlled degree and height passing through the point. This was one of the main ingredients of \cite{CVV_trans} where Checcoli, Veneziano and the second author prove a bound for $C$-torsion anomalous points of relative codimension one on a curve $C$.
We are going to use their  approximation theorem.

\begin{thm}[\cite{CVV_trans} Proposition 3.1 and 3.2] \label{thm:approximation}
Let $P$ be a point in a torsion variety $B \subset E^N$ and $k, s$ integers with  $k,s \in \{1,\ldots,N\}$ and $k \geq \dim B$. Then there are effective positive constants $c_3, c_4, c_5$ depending only on $E$, $N$, $k$ and $s$ such that for every real $T \geq 1$ there exists an abelian subvariety $H \subset E^N$ of codimension $s$ such that
\begin{eqnarray*}
\deg (H + P) & \leq & c_3 T, \\
h(H + P) & \leq & c_4 T^{1-\frac{N}{ks}} \hat{h}(P) + c_5 T.
\end{eqnarray*}
If $E$ is non CM, then the constants $c_3,c_4,c_5$ are explicit.
\end{thm}
Note that, in~\cite{CVV_trans}, the constants are explicit if $E$ is non CM. In a work in progress~\cite{CV_progress2}, S. Checcoli and E. Viada give explicit constants also in the CM case.

\section{Proof of the main theorem}
In this section, we denote by $V$  a weak-transverse subvariety of $E^N$ of codimension $2$. To prove our main Theorem~\ref{th:main}, we have to show that $V$ contains only finitely many maximal $V$-torsion anomalous varieties and that the normalized height of maximal $V$-torsion anomalous varieties is effectively bounded from above in terms of $V$, $E$ and $N$. Note that it is enough to show this for $V$ weak-transverse because a variety $V\subset E^N$ which is not weak-transverse is itself $V$-torsion anomalous. We show that there are only finitely many maximal $V$-torsion anomalous varieties of the following types:
\begin{enumerate}
\item
maximal torsion varieties contained in $V$,

\item
maximal $V$-torsion anomalous varieties which are not translates,

\item
maximal $V$-torsion anomalous points which are not torsion points,

\item
maximal $V$-torsion anomalous translates of positive dimension which are not torsion varieties.
\end{enumerate}
Clearly this covers all possible cases.

Note that, for a $V$-torsion anomalous variety $Y$, there is a unique minimal torsion variety $B \subset E^N$ such that $Y$ is an irreducible component of $V \cap B $. This is easily seen because every non-empty intersection of two torsion varieties in $E^N$ is again a torsion variety. Recall that the \emph{relative codimension} of $Y$ is
$\codim_Y B  = \dim B - \dim Y$ . Since $Y$ is $V$-torsion anomalous, we have
\[ \dim B - \dim Y = \codim Y - \codim B < \codim V = 2. \]
This means that any $V$-torsion anomalous subvariety $Y$ in $V$ is of relative codimension $0$ or $1$. Relative codimension $0$ only occurs if $Y$ is a torsion variety (type (1)). By the Manin-Mumford Conjecture (Theorem~\ref{con:MM}), we see that the maximal torsion subvarieties of $E^N$ contained in $V$ are exactly the irreducible components of the Zariski closure of the set of torsion points in $V$. Therefore, there are only finitely many $Y$ of type (1) and their normalized height is clearly trivially bounded.

For $Y$ of type (2), (3) or (4) the relative codimension is $1$.
 Previous results of Checcoli, Veneziano and the second author show that there are only finitely many $Y$ of type (2) and that their normalized height $h(Y)$ is effectively bounded~\cite[Theorem 5.1]{CVV_rendiconti}. In a subsequent paper the three authors cover the case of $Y$ of type (3). They show that the normalized height of the maximal $V$-torsion anomalous points of relative codimension one is effectively bounded~\cite[Theorem 1.1]{CVV_trans} and that these points are finitely many \cite[Corollary 1.2]{CVV_trans}.

It remains to show that there are only finitely many $Y$ of type (4) and that there is an effective upper bound for their normalized height  only depending on $V$, $E$ and $N$.

From now on, we denote by $Y$ a maximal $V$-torsion anomalous variety of type (4), that is
\begin{itemize}
\item
of the form $Y = H + p$ for a non-trivial abelian subvariety $H$ and a point $p$ in $ E^N$,

\item
an irreducible component of $V \cap B$ where $B \subset E^N$ is a torsion variety with $\dim B = \dim H + 1$.
\end{itemize}

\subsection{Bounded Height}
We first prove the following 
\begin{propo}\label{Height}
Under the same assumption as in Theorem \ref{th:main} we have that the maximal $V$-torsion anomalous varieties $Y$ of type (4) are of the  form $Y = H + p_1$ with $p_1$ lying in a torsion subvariety of dimension~$1$. In addition,  the height of $p_1$ is bounded as 
\[ \hat{h}(p_1) \leq C (\deg V + h(V))(\deg V)^{N - 1} \]
and the normalized height of $Y$ is bounded as
$$h(Y) \le 3 C  (\deg V + h(V))(\deg V)^{2^{\dim V} + N - 1}$$
for some effective constant $C$ depending only on $E$ and $N$. The constant can be made explicit.
\end{propo}

\begin{proof}
Note that $Y = H + p \subset B=B_0 + \zeta$ with $B_0$ an abelian subvariety and $\zeta$ a torsion point. Thus $0+p-\zeta$ and  $H$ are contained in $ B_0$. Moreover, there is a unique abelian subvariety $B_1 \subset E^N$ of dimension $1$ such that we have $B_0 = H + B_1$ and the Lie algebras of $B_1$ and $H$ are orthogonal to each other as subspaces of $\CC^N$ with the canonical Hermitian structure (we use the identification introduced in Subsection~\ref{subsec:absubvarieties}). This abelian subvariety $B_1$ is equal to the identity component of $H^\perp \cap B_0$ where $H^\perp$ is the orthogonal complement of $H$ in $E^N$ as defined in~\ref{subsec:absubvarieties}. Since $p$ lies in $B  = H + B_1 + \zeta$, we can write
\[ p = h + p_1 \]
with $h \in H$ and $p_1  \in B_1 + \zeta$. In particular, we have $Y = H + p_1$. We apply Theorem~\ref{thm:approximation} to the torsion variety $B_1 + \zeta$, the point $p_1$ lying in the torsion variety $B_1 + \zeta$ and the integers $k = 1$ and $s = N - 1$. By this theorem, for each $T \geq 1$, there is an abelian subvariety $H_1 \subset E^N$ of dimension $1$ such that
\begin{align}
\deg (H_1 + p_1) & \leq  c_3 T,  \label{ineq:deg_H1} \\
h(H_1 + p_1) & \leq  c_4 \frac{\hat{h}(p_1)}{T^{\frac{1}{N-1}}}  + c_5 T. \label{ineq:h_H1}
\end{align}

We claim that $Y$ is an irreducible component of the intersection $V \cap (H + H_1 + p_1)$. To prove this, we note that we have $\dim Y \geq \dim (H + H_1 + p_1) - 1$, where equality holds if and only if $H_1$ is not contained in $H$. Therefore, if $Y$ is no irreducible component of $V \cap (H + H_1 + p_1)$, then $H_1$ is not contained in $H$ and $H + H_1 + p_1$ is contained in $V$. In this case, the translate $H + H_1 + p_1$ would be contained in the intersection $V \cap (B + H_1 + \zeta)$ with
\begin{align*}
 \codim (H + H_1 + p_1) & = \codim Y - 1 = \codim B \\
 & \leq \codim (B + H_1) + 1 < \codim (B + H_1) + \codim V, 
\end{align*}
hence $H + H_1 + p_1 = Y + H_1$ would be contained in a $V$-torsion anomalous subvariety contradicting the maximality of $Y$. This proves that $Y$ is an irreducible component of $V \cap (H + H_1 + p_1)$.

Now we apply the arithmetic B\'ezout theorem to $Y \subset V \cap (H + H_1 + p_1)$ and get
\begin{equation} \label{ineq:bezout}
\begin{aligned}
h(Y) & \leq \deg (V) h(H + H_1 + p_1) + \deg (H + H_1 + p_1) h(V) \\
& + \frac{3^N \log 2}{2} \cdot \deg V \cdot \deg (H + H_1 + p_1).
\end{aligned}
\end{equation}
We recall that we have $Y = H + p_1 = H + p'_1 + \zeta$ with $p'_1 \in B_1 \subset H^{\perp}$ and $\zeta$ the torsion point specified above. Therefore, using Lemma  \ref{lem:Philippon}, the Zhang Inequality  \eqref{ineq:zhang} and the relation~\eqref{ineq:mu2}, we get the inequality 
\begin{equation} \label{ineq:hp_1}
\hat{h}(p_1) = \hat{h}(p'_1) = \hat{\mu}(H + p'_1) = \hat{\mu}(Y) \leq \frac{\mu(Y)}{3} + c_2 \leq \frac{h(Y)}{3 \deg (H)} + c_2.
\end{equation}
We now bound the right hand side of this inequality from above using~\eqref{ineq:bezout} by estimating $h(H + H_1 + p_1)$ and $\deg(H + H_1 + p_1)$ from above. By~\eqref{ineq:zhang} and \eqref{ineq:mu1}, we get
\[ h(H+H_1+p_1) \leq \deg(H + H_1) \cdot (1 + \dim (H + H_1)) \cdot (3 \hat{\mu}(H + H_1 + p_1) + c_1). \]
Note that we have $\hat{h}(\zeta + P) = \hat{h}(P)$ for all $P \in H_1 + p_1$ and torsion points $\zeta$ in $H$. Since the torsion points of $H$ are Zariski dense in $H$ and  by the definition of the N\'eron-Tate essential minimum we get
\[ \hat{\mu}(H + H_1 + p_1) \leq \hat{\mu}(H_1 + p_1). \]
Together with the previous inequality and the Zhang Inequality \eqref{ineq:zhang} and the relation \eqref{ineq:mu2}, we get
\[ h(H+H_1+p_1) \leq \deg(H + H_1) \cdot (N + 1) \cdot \left(\frac{h(H_1 + p_1)}{\deg (H_1)} + 3c_2 + c_1\right). \]
Combined with \eqref{ineq:h_H1}, \eqref{ineq:deg_H1} and \eqref{ineq:sum}, this gives
\begin{equation} \label{ineq:h}
\frac{h(H+H_1+p_1)}{\deg (H)} \leq c_6 \cdot (N + 1)\cdot \left( c_4 \frac{\hat{h}(p_1)}{T^{\frac{1}{N-1}}} + (c_5  + 3c_2c_3 + c_1c_3)T \right).
\end{equation}
 Furthermore, by \eqref{ineq:sum} and \eqref{ineq:deg_H1}, we have
\begin{equation} \label{ineq:deg}
 \frac{\deg(H+H_1+p_1)}{\deg (H)} = \frac{\deg(H+H_1)}{\deg (H)} \leq c_3c_6 T.
\end{equation}
Combining \eqref{ineq:h}, \eqref{ineq:deg} with \eqref{ineq:bezout}, we get
\begin{equation} \label{ineq:hoverd}
 \frac{h(Y)}{\deg(H)}  \leq \frac{c_7 \deg V}{T^{\frac{1}{N-1}}} \hat{h}(p_1) +  (c_8 + c_9 \deg V + c_{10} h(V)) T
\end{equation}
with the effective positive constants
\begin{align*} 
 c_7 & := \mathrm{max} \{1, c_4c_6(N + 1)\}, \\
 c_8 & := c_5 c_6, \\
 c_9 & := c_3 c_6 \left( (3c_2 + c_1)(N + 1)+\frac{3^N \log 2}{2}\right), \\
 c_{10} & := c_3 c_5 c_6
\end{align*}
depending only on $E$ and $N$. Together with \eqref{ineq:hp_1} this gives
\begin{equation}
\left( 3 - \frac{c_7 \deg V}{T^{\frac{1}{N-1}}} \right) \hat{h}(p_1) \leq (c_8 + c_9 \deg V + c_{10} h(V)) T + 3c_2.
\end{equation}
We now choose $T := c_7^{N-1} (\deg V)^{N-1} \geq 1$ and get
\begin{equation} \label{ineq:mu_y}
\hat{\mu}(Y) = \hat{h}(p_1) \leq C(\deg V + h(V))(\deg V)^{N - 1}
\end{equation}  
with $C :=  \frac12 c_7^{N-1} (\mathrm{max} \{ c_9, c_{10} \} +  c_8 + 3 c_2)$. This concludes the proof of the first part of the proposition. 

 For the normalized height $h(Y)$, 
combining~\eqref{ineq:hoverd},  \eqref{ineq:mu_y} and our choice of $T$, we get the upper bound
\[ h(Y) \leq 3C \deg Y \cdot (\deg V + h(V))(\deg V)^{N - 1}. \]

Note that $Y$ is a maximal translate contained in $V$ because all maximal $V$-torsion anomalous translates are maximal translates contained in $V$ by~\cite[Lemma 7.1]{CVV_rendiconti}. Using Lemma 2 from~\cite{BZ} and the inductive construction in its proof, we can therefore bound uniformly the degree of $Y$ in terms of $V$ by
\begin{equation} \label{ineq:deg_BZ}
\deg Y \leq (\deg V)^{2^{\dim V}}.
\end{equation}
Hence we have
\[ h(Y) \leq 3C (\deg V + h(V))(\deg V)^{2^{\dim V} + N - 1} .\]

In addition, the constant $C$  can be made  explicit, because the constants $c_i$ for $i=1,\dots 6$ are  explicit in the non-CM case, and in a preprint~\cite{CV_progress2} are made explicit also for the CM case.
\end{proof}

It remains to show the finiteness.
 
\subsection{Finiteness}

\begin{propo}
Under the same assumption  as in Theorem \ref{th:main}, we have that there exist only finitely many maximal $V$-torsion anomalous varieties of type (4).\end{propo}

\begin{proof}[Proof of the Proposition]
We recall that  all maximal $V$-torsion anomalous translates in $E^N$ are of the form $Y = H + p$ with $H$ an abelian subvariety and 
\[ \deg H = \deg Y \leq (\deg V)^{2^{\dim V}}. \]
As recalled in Section~\ref{subsec:absubvarieties} there are only finitely many abelian subvarieties $H$ of $E^N$ with bounded degree.
Thus to prove the proposition it is sufficient to show that, for a fixed abelian subvariety $H$ of $E^N$, there are only finitely many translates of $H$ which are maximal $V$-torsion anomalous.

The proof relies on the following non-density theorem by E. Viada.

\begin{thm}[\cite{Evelina_IMRN} Theorem 1.1 (i)] \label{th:IMRN}
Let $V \subset E^N$ be a weak-transverse subvariety of dimension $d$ and $T \geq 0$ a real number. Then the set
\[ S_{d+1}(V_T) := \{p \in V\,:\, \hat{h}(p) \leq T \} \cap \bigcup_{\cod B \geq d+1} B \]
where $B$ runs over all torsion varieties of codimension at least $d+1$ is not Zariski dense in $V$.
\end{thm}

We use induction on the dimension of $V$. If $V$ is a weak-transverse curve of codimension $2$, then there are no $V$-torsion anomalous translates of positive dimension, therefore the statement follows by part (1), (2) and (3)  discussed at the beginning of this section. 

Assume it is proven for $\dim V=d-1$ and $\codim V=2$, we then show it for $\dim V=d$ and $\codim V=2$.  We assume by contradiction that there are infinitely many maximal $V$-torsion anomalous translates $Y_i = H + p_i$ of $H$. By~Theorem \ref{Height} we can take $p_i \in B_i$ for a torsion variety $B_i$ of dimension one. In addition, $\hat{h}(p_i) \leq T$ for $T := 3 C (\deg V + h(V))(\deg V)^{N - 1}$. Therefore, the union of all points in $p_i+{\rm Tor}_H$ is contained in $S_{N-1}(V_T)$ and by Theorem~\ref{th:IMRN} it is not Zariski dense in $V$. Thus the closure $X$ of  $\cup_i (p_i+H) = \cup_i Y_i$  is strictly contained in $V$, so $\dim X<\dim V$.   Since we assumed that there are infinitely many $Y_i$, we also have $\dim X > \dim H$. 
Let $A$ be the minimal torsion variety containing $X$. Since the union $\cup_i (p_i+{\rm Tor}_H)$ is dense in $X$ and a subset of $S_{N-1}(X_T) \subset S_{\dim X+1}(X_T)$, by Theorem~\ref{th:IMRN} the subvariety $X$ cannot be weak-tranverse in $E^N$. Thus $A\not=E^N$. Now we consider an irreducible component $Z$ of $V \cap A$ which contains $X$. By the maximality of $H+p_i$, the variety $Z$ cannot be $V$-torsion anomalous thus $\codim_A Z=2$ and $\dim Z = \dim A - 2 < N - 2 = \dim V = d$. 
 Consider the weak-transverse variety $Z$ embedded in $A$. Note that the infinitely many $H+p_i$ are subvarieties of $Z$.
Moreover, by an easy check of the codimensional equation, we see that  the $H+p_i$ are $Z$-torsion anomalous in $A$, and by inductive hypothesis finitely many in contradiction to our assumption. 

\end{proof}

\section*{Acknowledgements}
We thank the DFG for  financial support.

\def\cprime{$'$}
\providecommand{\bysame}{\leavevmode\hbox to3em{\hrulefill}\thinspace}
\providecommand{\MR}{\relax\ifhmode\unskip\space\fi MR }
\providecommand{\MRhref}[2]{%
  \href{http://www.ams.org/mathscinet-getitem?mr=#1}{#2}
}
\providecommand{\href}[2]{#2}

\section*{}
\noindent 
Patrik Hubschmid:
Mathematisches Institut, 
Georg-August-Universit\"{a}t G\"{o}ttingen,
Bunsenstra\ss e 3-5,
37073 G\"{o}ttingen, Germany.
Email: patrikh@gmx.ch
\medskip\\
Evelina Viada:
Department of Mathematics, ETH Zurich, R\"amistrasse 101, 8092 Zurich, Switzerland.
Email: evelina.viada@math.ethz.ch

\end{document}